\documentclass[11pt]{article}
\usepackage{amsmath,amssymb,amsthm}
\usepackage{geometry}
\usepackage{graphicx}
\usepackage[numbers,sort&compress]{natbib}

\theoremstyle{plain}
\newtheorem{theorem}{Theorem}
\newtheorem{lemma}{Lemma}
\newtheorem{definition}{Definition}
\theoremstyle{remark}
\newtheorem{remark}{Remark}

\newcommand{\fo}{f_1^*}   
\newcommand{\ft}{f_2^*}   
\newcommand{\fy}{f_y^*}   
\newcommand{\gx}{g_x^*}   
\newcommand{\gy}{g_y^*}   
\newcommand{\foo}{f_{11}^*}
\newcommand{\fot}{f_{12}^*}
\newcommand{\ftt}{f_{22}^*}
\newcommand{\foy}{f_{1y}^*}
\newcommand{\fty}{f_{2y}^*}
\newcommand{\fooo}{f_{111}^*}
\newcommand{\foot}{f_{112}^*}
\newcommand{\fott}{f_{122}^*}
\newcommand{\fttt}{f_{222}^*}

\title{Cusped singularities organize mixed-mode oscillations in
mutually inhibitory slow-fast systems}
\author{Morten Gram Pedersen\textsuperscript{a,b,*} \\
\small \textsuperscript{a}Department of Information Engineering, University of Padova, 35131 Padova, Italy \\
\small \textsuperscript{b}Padova Neuroscience Center, University of Padova, 35129 Padova, Italy \\
\small \textsuperscript{*}Corresponding author: mortengram.pedersen@unipd.it}
\date{}

\begin{document}
\maketitle

\begin{abstract}
Mutual inhibition is a common motif in neural systems. Here, we establish that cusped
singularities---folded singularities located at cusp points of critical manifolds---provide a
universal organizing mechanism for mixed-mode oscillations (MMOs) in coupled slow-fast
systems with mutual inhibition.
We show that the geometric setup of these systems generically satisfies the conditions
required by established geometric singular perturbation theory and blow-up methods,
guaranteeing that such cusped singularities yield small-amplitude oscillations (SAOs).
MMOs appear from the SAOs combined with an appropriate return mechanism.
Further, we show that the geometric presence of a cusped singularity is
strictly related to occurrence of a nearby singular Hopf bifurcation.
We demonstrate the efficacy of this framework in two distinct neuronal models: the Curtu
rate model of mutually inhibitory neural populations and coupled Morris-Lecar neurons
with synaptic inhibition.
In both cases, pushing the full system equilibrium near the cusped singularity triggers
SAOs as the system passes near the cusp and approaches a full-system saddle-focus related
to the singular Hopf bifurcation.
Large-amplitude oscillations appear as the system spirals away from the saddle-focus,
leading to MMOs, which may exhibit distinctive alternating patterns, in contrast to
standard saddle-node induced MMOs.
Our results establish cusped singularities as a generic, biologically relevant mechanism
for complex oscillatory dynamics in inhibitory neural networks as well as for other
inhibitory slow-fast systems.

\end{abstract}

\section{Introduction}
Mixed-mode oscillations (MMOs)---complex rhythms alternating between small-amplitude
oscillations (SAOs) and large-amplitude oscillations (LAOs)---are ubiquitous in neural
systems operating on multiple time scales \cite{desroches2012}.
Classical mechanisms for MMOs include folded nodes, canard solutions, and singular Hopf
bifurcations \cite{desroches2012, wechselberger2005, krupa2010, guckenheimer2008,
battaglin2021}.
Recently, Kristiansen and Pedersen \cite{kristiansen2023} identified a novel mechanism
to explain MMOs in coupled FitzHugh-Nagumo oscillators with repulsive coupling
\cite{pedersen2022}: the cusped singularity, a folded singularity located at a cusp point
of the critical manifold.

Here we develop a comprehensive theoretical framework showing that cusped singularities
provide a universal organizing mechanism for MMOs in mutually inhibitory neuronal
systems.
We consider a general class of slow-fast systems with symmetric coupling, and prove that
the presence of a cusped singularity implies the occurrence of a singular Hopf
bifurcation in the full system.
This result establishes a direct correspondence between the geometry of the critical
manifold (cusp point) and the dynamic bifurcation (singular Hopf), providing a powerful
tool for analyzing MMOs in diverse neural models.
Our framework is then applied to two biophysically distinct models of neuronal activity: 
(i) the Curtu rate model
\cite{curtu2010}, describing mutually inhibitory neural populations with sigmoidal
activation; and (ii) coupled Morris-Lecar neurons \cite{morris1981} with fast synaptic
inhibition.
Despite their different mathematical structures and biological
interpretations,
both models exhibit cusped singularities that organize MMOs.
This universality underscores the biological relevance of the mechanism across different
levels of neural modeling, but also points towards the wider applicability of the framework 
to other types of symmetric systems with inhibitory coupling.

The paper is organized as follows. Section~\ref{sec:framework} establishes the general
theoretical framework, deriving conditions for cusped singularities, performing the center
manifold reduction that underpins the SAO count, and proving their relation to singular
Hopf bifurcations.
Section~\ref{sec:curtu} applies the theory to the Curtu model,
Section~\ref{sec:ml} to the Morris-Lecar model, and Section~\ref{sec:discussion}
discusses implications for neural dynamics and beyond.

\section{General framework for cusped singularities}
\label{sec:framework}

Consider a slow-fast system with two identical mutually inhibitory units:
\begin{align}
\dot{x}_{1} &= f(x_{1},x_{2},y_{1}), \quad \dot{y}_{1} = \epsilon g(x_{1},y_{1}),
\nonumber \\
\dot{x}_{2} &= f(x_{2},x_{1},y_{2}), \quad \dot{y}_{2} = \epsilon g(x_{2},y_{2}),
\label{eq:full}
\end{align}
where $\epsilon\ll1$.
The fast dynamics depends on the local fast variable, the coupled fast variable, and the
local slow variable.
We assume inhibitory coupling, meaning $\partial_{x_{j}}f(x_{i},x_{j},y_{i})<0$
(increasing the coupled variable suppresses the local variable).
The system~\eqref{eq:full} is symmetric under the exchange
$\mathcal{T}:(x_{1},x_{2},y_{1},y_{2})\mapsto(x_{2},x_{1},y_{2},y_{1})$.
We will use the notation $(x_i,x_j,y_i,y_j)\in
\mathbb{R}^2_{\mathrm{fast}}\times \mathbb{R}^2_{\mathrm{slow}}$.

\subsection{Critical manifold geometry}
The critical manifold $\mathcal{C}$ is defined by
\begin{equation}
f(x_{i},x_{j},y_{i}) = 0, \quad i=1,2.
\label{eq:crit}
\end{equation}
Assuming $\partial_{y}f\ne0$ (solvability for the implicit function theorem), we solve
\eqref{eq:crit} for the slow variables $y_{i} = Y(x_{i},x_{j})$.
This parametrizes $\mathcal{C}$ as a graph
$(y_{1},y_{2}) = F(x_{1},x_{2}) = (Y(x_{1},x_{2}), Y(x_{2},x_{1}))$.

The Jacobian is given by
\begin{equation}
DF(x_{1},x_{2}) = \begin{pmatrix}
\partial_{x_{1}}Y(x_{1},x_{2}) & \partial_{x_{2}}Y(x_{1},x_{2}) \\
\partial_{x_{1}}Y(x_{2},x_{1}) & \partial_{x_{2}}Y(x_{2},x_{1})
\end{pmatrix} =
\begin{pmatrix} Y_1(x_{1},x_{2}) & Y_2(x_{1},x_{2}) \\
Y_2(x_{2},x_{1}) & Y_1(x_{2},x_{1}) \end{pmatrix}
\label{eq:DF}
\end{equation}
where $Y_k$ means differentiation with respect to the $k$th argument.
The projection $\pi:\mathcal{C}\rightarrow\mathbb{R}^2_{\mathrm{slow}}$ is singular where
$\det(DF)=0$, which defines a \emph{fold} on $\mathcal{C}$.
By symmetry $\mathcal{T}$ and implicit differentiation of \eqref{eq:crit}, at symmetric
points $p^{*}=(x^{*},x^{*},y^{*},y^{*})$ with $y^{*}=Y(x^{*},x^{*})$ we have
\begin{equation}
DF|_{p^{*}} = -\frac{1}{\fy}
\begin{pmatrix} \fo & \ft \\ \ft & \fo \end{pmatrix}
\label{eq:DF_eval}
\end{equation}
where $\fo := \partial_{x_{1}}f(x^{*},x^{*},y^{*})$,
$\ft := \partial_{x_{2}}f(x^{*},x^{*},y^{*})$, and
$\fy := \partial_{y}f(x^{*},x^{*},y^{*})$.

The eigenvalues of $DF|_{p^*}$ are $-(\fo\pm\ft)/\fy$.
The fold condition $\det(DF)=0$ yields two possibilities.
Here, we rigorously consider the fold condition 
\begin{equation}
\fo = \ft.
\label{eq:fold_new}
\end{equation}
At this symmetric fold point, the kernel of $DF$ is spanned by the antisymmetric vector
$(1, -1)^T$.
Note that by the assumption of inhibitory coupling $\partial_{x_2}f <0$, so that also $\partial_{x_1}f<0 $, and hence also $Y_1\neq 0$,
along the  fold in a neighborhood of $p^*$.
At symmetric points on the fold, the gradient of the determinant is
$\nabla \det DF(x^*,x^*) = Y_1(x^*,x^*)[Y_{11}(x^*,x^*)-Y_{22}(x^*,x^*)] (1,1)^T$,
which is perpendicular to $\ker(DF)$.
Hence, the fold curve is tangent to the antisymmetric subspace.
For the fold curve to be regular, the gradient must be non-vanishing, which yields the
non-degeneracy constraint
\begin{equation}
Y_{11}(x^{*},x^{*}) - Y_{22}(x^{*},x^{*}) \ne 0. \label{eq:cusp_new}
\end{equation}
Using implicit differentiation, the equivalent constraint expressed in partial derivatives
of the fast dynamics function $f$ is
\begin{equation}
D^*=\foo - \ftt
- 2\left( \frac{\fo}{\fy} \right)
\left( \foy - \fty \right) \ne 0,
\label{eq:cusp_f}
\end{equation}
where we use -- here and in the following -- the compact notation $\foo = \partial^2_{x_1x_1}f(x^*,x^*,y^*)$, $\foy = \partial^2_{x_1y}f(x^*,x^*,y^*)$, etc.

\begin{definition}[Non-degenerate cusped singularity]
\label{def:cusp}
A symmetric point $p^{*}=(x^{*},x^{*},y^{*},y^{*})\in\mathcal{C}$ is a
\emph{non-degenerate cusped singularity} if it satisfies the fold condition
\eqref{eq:fold_new} and the non-degeneracy constraint \eqref{eq:cusp_new}
(or equivalently \eqref{eq:cusp_f}).
\end{definition}

\subsection{From definition to cusp normal form}
We demonstrate how Definition~\ref{def:cusp} implies the standard cusp curve
$z^{2} \propto w^{3}$ for the critical manifold projection along the symmetric subspace.
Near the singularity $p^{*}$, we define symmetry-adapted coordinates
\begin{align}
v &= \tfrac{1}{2}(x_{1}+x_{2})-x^{*}, \quad u = \tfrac{1}{2}(x_{1}-x_{2}),
\label{eq:coord_vu} \\
w &= \tfrac{1}{2}(y_1+y_2) - y^{*}, \quad z = \tfrac{1}{2}(y_1-y_2).
\label{eq:coord_wz}
\end{align}
In these coordinates, the critical manifold $\mathcal{C}$ is given by a graph
$(w,z)=({W}(v,u), {Z}(v,u))$ where
${W}(0,0)={Z}(0,0)=0$.
Crucially, by construction, ${W}(v,u)$ is an even function of $u$, while
${Z}(v,u)$ is an odd function of $u$.
The odd symmetry enforces ${Z}(v,0) = 0$ and forces all even partial derivatives
of ${Z}$ with respect to $u$ to vanish at $u=0$.
In addition, all pure $v$-derivatives of ${Z}$ vanish identically at the origin.
Differentiating ${Z}(v,u)$ with respect to $u$ and evaluating at the origin yields
\begin{equation}
\partial_{u}{Z}(0,0) = \partial_{x_{1}}Y(x^{*},x^{*}) -
\partial_{x_{2}}Y(x^{*},x^{*}) = -\frac{1}{\fy}(\fo - \ft).
\end{equation}
Our fold condition \eqref{eq:fold_new}, which states $\fo = \ft$, is
therefore strictly equivalent to $\partial_{u}{Z}(0,0) = 0$.
This shows the projection is degenerate along the $u$-direction.

Because $\partial_{u}^{2}{Z}(0,0) = 0$ trivially by the odd symmetry, the
geometry of the fold curve is governed by the mixed partial derivative.
Differentiating $\partial_u {Z}$ with respect to $v$ yields
\begin{equation}
\partial_{v}\partial_{u}{Z}(0,0) = Y_{11}(x^{*},x^{*}) - Y_{22}(x^{*},x^{*}).
\end{equation}
Thus, our derived non-degeneracy condition \eqref{eq:cusp_new} guarantees that
$\partial_{v}\partial_{u}{Z}(0,0) \ne 0$.
Applying the fold and non-degeneracy conditions to the Taylor expansion of
${Z}(v,u)$, and utilizing its odd symmetry, all linear and pure even-order terms
vanish, leaving
\begin{equation}
{Z}(v,u) = B v u + A u^{3} + \mathcal{O}(u^{5}, v^{2}u),
\label{eq:Zexp}
\end{equation}
where $A = \frac{1}{6}\partial_{u}^{3}{Z}(0,0)$ and
$B = \partial_{v}\partial_{u}{Z}(0,0) \ne 0$.
We assume the generic transversality condition $A \ne 0$ to ensure the singularity forms
a standard cusp rather than a higher-order flat manifold.
The fold curve $\mathcal{F}$ satisfies $\partial_u {Z}(v,u) = 0$, giving
$v = -\frac{3A}{B}u^2 + \mathcal{O}(u^4)$, i.e., $z \approx -2A u^3$.
Simultaneously, because ${W}(v,u)$ is an even function of $u$, its expansion
gives $w \approx Cv + Du^2$.
Substituting $v \propto u^2$, we find $w \propto u^2$.
Eliminating the local variable $u$ yields the classic cusp geometry
\begin{equation}
z^{2} \propto w^{3}.
\label{eq:cusp_geom}
\end{equation}

\subsection{Center manifold reduction and the reduced slow-fast system}
\label{sec:cm}

We now derive the center manifold reduction that underpins the SAO analysis. This
section makes explicit the conditions on $f$ and $g$ required for the blow-up results
of \cite{kristiansen2023} to apply to the general system \eqref{eq:full}, and provides
the key structural lemma.

We work in the symmetry-adapted coordinates \eqref{eq:coord_vu}--\eqref{eq:coord_wz}
and augment \eqref{eq:full} by $\dot\epsilon=0$ to treat the extended system.
Let
\begin{align}
\Omega &:= - \frac{\fy}{2\fo} D^* = (\foy-\fty) - (\foo-\ftt)\frac{\fy}{2\fo},
\label{eq:Dy} \\
\Gamma &:= \frac{\fooo-3\foot+3\fott-\fttt}{6}
- \frac{(\foo-\ftt)(\foo-2\fot+\ftt)}{4\fo},
\label{eq:Acal}
\end{align}
and note that the non-degeneracy condition \eqref{eq:cusp_f} gives $\Omega\neq 0 $.

\begin{lemma}[Center manifold reduction]
\label{lem:cm}
Assume that $p^*$ is a non-degenerate cusped singularity (Definition~\ref{def:cusp}) and that $\Gamma \ne 0$.
Then the extended system \eqref{eq:full} augmented by $\dot\epsilon=0$, written in
coordinates $(v,u,w,z,\epsilon)$ centered at $p^*$, possesses a smooth attracting
four-dimensional center manifold
\begin{equation}
M_a:\quad v = h(u,w,z,\epsilon),
\label{eq:cm}
\end{equation}
which is even in $(u,z)$ and satisfies
\begin{equation}
h(u,w,z,\epsilon) = -\frac{\fy}{2\fo}\,w
- \frac{\foo-2\fot+\ftt}{4\fo}\,u^2
+ \mathcal{O}(\epsilon,\,uz,\,u^4,\,w^2,\,z^2).
\label{eq:h_expansion}
\end{equation}
The reduced dynamics on $M_a$ takes the form
\begin{equation}
\begin{aligned}
\dot{u} &= \fy z + \Omega\, uw + \Gamma\, u^3 + \mathcal{O}(4), \\
\dot{w} &= \epsilon\bigl(g_0^*+ \nu_{\mathrm{eff}}\, w  
+ \rho_{\mathrm{eff}}\, u^2 + \mathcal{O}(w^2, uz, u^4, z^2)\bigr), \\
\dot{z} &= \epsilon\bigl(\gx u + \gy z + \mathcal{O}(u^3, uw, wz)\bigr),
\end{aligned}
\label{eq:reduced_3d}
\end{equation}
where
\begin{equation}
g_0^* := g(x^*,y^*), \qquad \nu_{\mathrm{eff}} := \gy - \frac{\gx\fy}{2\fo},
\qquad
\rho_{\mathrm{eff}} := -\frac{\gx(\foo-2\fot+\ftt)}{4\fo} +\frac12 g_{xx}^*.
\label{eq:eff_params}
\end{equation}
The system \eqref{eq:reduced_3d} is slow-fast with one fast variable $u$ and two slow
variables $(w,z)$.
Its critical manifold $\mathcal{S}$, defined by $\dot u = 0$, is given by
$z = Q(u,w)$ where $Q$ is odd in $u$ and satisfies
\begin{equation}
Q(u,w) = -\frac{\Omega}{\fy}\,uw - \frac{\Gamma}{\fy}\,u^3
+ \mathcal{O}(u^5,\,uw^2,\,u^3 w).
\label{eq:Q}
\end{equation}
The point $(u,w,z)=(0,0,0)$ is a cusp singularity of $\mathcal{S}$ in the sense of
\cite{kristiansen2023}.
\end{lemma}

\begin{proof}
\textbf{Step 1: Taylor expansion.}
Define $F^+(v,u,w,z):=f(x^*+v+u,\,x^*+v-u,\,y^*+w+z)$ and
$F^- := F^+(v,-u,w,-z)$ (by $\mathcal{T}$-symmetry of $f$).
Then $\dot v = (F^++F^-)/2$ is even in $(u,z)$ and
$\dot u = (F^+-F^-)/2$ is odd in $(u,z)$.
A direct Taylor expansion gives:
\begin{equation}
\dot v = 2\fo v + \fy w
+ \tfrac{\foo-2\fot+\ftt}{2}\,u^2
+ (\foy+\fty)\,vw + \mathcal{O}(v^2, w^2, z^2, uz, u^4).
\label{eq:vdot}
\end{equation}
For the $u$-equation, by symmetry only 
odd-in-$(u,z)$ terms are non-zero, and we get:
\begin{equation}
\begin{split}
\dot u = \fy z + (\foo-\ftt)\,vu + (\foy-\fty)\,uw
&+ \tfrac{\fooo-3\foot+3\fott-\fttt}{6}\,u^3 \\ 
&+ \mathcal{O}(u^5, v^2u, w^2u, vwu, vz, wz, u^2z, z^3).
\label{eq:udot}
\end{split}
\end{equation}
The slow equations are given by the Taylor expansion of $g(x^*+v\pm u, y^*+w\pm z)$:
\begin{equation}
\begin{aligned}
\dot w &= \epsilon\left(g_0^* + \gx v + \gy w + \frac{g_{xx}^*}{2}(v^2 + u^2) + g_{xy}^*(vw + uz) + \frac{g_{yy}^*}{2}(w^2 + z^2) + \mathcal{O}(v^3, u^2v, u^4)\right), \\
\dot z &= \epsilon\left(\gx u + \gy z + g_{xx}^*vu + g_{xy}^*(vz+uw) + g_{yy}^*wz + \mathcal{O}(u^3, uv^2)\right).
\end{aligned}
\label{eq:wzdot}
\end{equation}

\smallskip
\textbf{Step 2: Center manifold existence.}
At $p^*$ with $\epsilon=0$, the linearization of the extended system has eigenvalue
$2\fo<0$ in the $v$-direction (by $\fo<0$ from the fold condition and
inhibitory coupling) and eigenvalue $0$ in the $u$, $w$, $z$, $\epsilon$ directions.
By the center manifold theorem \cite{carr1981}, the extended system possesses a smooth
four-dimensional center manifold $M_a: v=h(u,w,z,\epsilon)$.
The $\mathcal{T}$-symmetry of \eqref{eq:full} forces $h$ to be even in $(u,z)$.

\smallskip
\textbf{Step 3: Expansion of $h$.}
The invariance equation $\dot v|_{v=h}
= (\partial_u h)\dot u + (\partial_w h)\dot w + (\partial_z h)\dot z$
at $\epsilon=0$ evaluated on the center manifold $v = h_{0w}w + h_{uu}u^2 + \mathcal{O}(\epsilon, uz, u^4, w^2, z^2)$ gives:
At order~$w$ we have $2\fo\,h_{0w} + \fy = 0$, i.e.,
\[ 
h_{0w} = -\frac{\fy}{2\fo}.
\]
At order $u^2$, since $\partial_u h = 2h_{uu}u + \mathcal{O}(z, u^3, uw)$ and
$\dot u = \fy z + \mathcal{O}(u^3, uw, vu)$, the right-hand side is 
$(\partial_u h)\dot u = 2h_{uu}\fy uz + \mathcal{O}(u^4, u^2w, z^2)$, which contributes only at
weight $\geq 3$ in the blowup sense. Equating the $u^2$ terms of the left-hand side to zero gives
\[
h_{uu} = -\frac{\foo-2\fot+\ftt}{4\fo},
\]
establishing \eqref{eq:h_expansion}.

\smallskip
\textbf{Step 4: Reduced system on $M_a$.}
Substituting $v=h$ into \eqref{eq:udot} and using
$h = -(\fy/2\fo)w + h_{uu}\,u^2 + \mathcal{O}(uz, u^4, w^2, z^2)$:
\begin{equation*}
\begin{split}
\dot u &= \fy z
+ (\foo-\ftt)\!\left(-\frac{\fy}{2\fo}w + h_{uu}\,u^2\right)\!u
+ (\foy-\fty)\,uw
+ \tfrac{\fooo-3\foot+3\fott-\fttt}{6}\,u^3\\
&\hspace{10.5cm}+ \mathcal{O}(u^2 z, u^5, uw^2, uz^2, wz, z^3) \\
&= \fy z
+ \!\underbrace{\left[(\foy-\fty)
- (\foo-\ftt)\frac{\fy}{2\fo}\right]}_{\Omega}\!uw \\
&\qquad\qquad\qquad
+ \!\underbrace{\left[\frac{\fooo-3\foot+3\fott-\fttt}{6}
- \frac{(\foo-\ftt)(\foo-2\fot+\ftt)}{4\fo}\right]}_{\Gamma}\!u^3 +  \mathcal{O}(4). 
\end{split}
\end{equation*}
Substituting $v=h$ into the $\dot w$ equation in \eqref{eq:wzdot}:
\begin{align*}
\dot w &= \epsilon\!\left[g_0^* + \gx \!\left(-\frac{\fy}{2\fo}w + h_{uu}\,u^2\right)\! + \gy w + \frac{1}{2}g_{xx}^* u^2 + \mathcal{O}(w^2, uz, u^4, z^2) \right] \\
&= \epsilon\!\left[g_0^* + \!\left(\gy-\frac{\gx\fy}{2\fo}\right)\!w
+ \!\left( \gx h_{uu} + \frac{1}{2}g_{xx}^* \right)\! u^2 + \mathcal{O}(w^2, uz, u^4, z^2) \right] \\
&= \epsilon\bigl(g_0^* +\nu_{\mathrm{eff}}\,w + \rho_{\mathrm{eff}}\,u^2 + \mathcal{O}(w^2, uz, u^4, z^2)\bigr).
\end{align*}
This gives the reduced system \eqref{eq:reduced_3d}.

\smallskip
\textbf{Step 5: Cusp structure of $\mathcal{S}$.}
Setting $\dot u=0$ in \eqref{eq:reduced_3d} gives $z=Q(u,w)$ as in \eqref{eq:Q}.
By $\mathcal{T}$-symmetry, $Q$ is odd in $u$.
The fold set $\partial_u Q = 0$ yields
$w = -(3\Gamma/\Omega)\,u^2 + \mathcal{O}(u^4)$,
which is a smooth curve since $\Omega\ne0$.
Substituting back gives $z\approx 2(\Gamma/\fy)u^3$  
so $z^2 \approx \frac{- 4 \Omega^3}{27 (\fy)^2 \Gamma} w^3$, i.e, 
$z^2\propto w^3$: the cusp normal form, matching \cite[Lemma~3]{kristiansen2023}.

\end{proof}

\begin{remark}\label{rem1}
If $\Gamma/\Omega > 0$ the cusp opens up in the $w<0$ direction. 
Further, if $\Omega/\fy<0$ then the central sheet (present for $w<0$) of $\mathcal{S}$ is attracting.
These conditions are met if and only if $\Gamma$ and $\Omega$ have the same sign, which is different from the sign of $\fy$ .
For the FHN model of \cite{kristiansen2023}, one verifies that
$\fy=-1$, $\Omega = -3v_s/g > 0$ (since $v_s>0$, $g<0$) and
$\Gamma = -(9v_s^2+g)/g = -(9-5g)/g>0$.
Hence, $\Gamma/\Omega > 0$ and $\Omega/\fy<0$, confirming the cusp geometry outlined above, see Figure 11 in \cite{kristiansen2023}.
\end{remark}

\begin{remark}\label{rem2}
If $\Gamma/\Omega < 0$ the cusp opens up in the $w>0$ direction and the central sheet is attracting if $\Omega/\fy>0$. 
These conditions are equivalent to $\Gamma$ and $\fy$ having opposite signs and $\Omega$ and $\fy$ having the same sign.
\end{remark}

\subsection{Singular Hopf bifurcation and the cusped saddle-node}
\label{sec:hopf}

The full system \eqref{eq:full} has a symmetric equilibrium $p^{*}\in \mathcal{C}$ when
$g(x^{*},y^{*})=0$.
Using the symmetry eigenbasis, the Jacobian  $J$ at $p^*$ block-diagonalizes into symmetric ($J_s$) and
antisymmetric ($J_a$) components
\begin{equation}
J_{s} = \begin{pmatrix} \fo+\ft & \fy \\
\epsilon \gx & \epsilon \gy \end{pmatrix}, \quad
J_{a} = \begin{pmatrix} \fo-\ft & \fy \\
\epsilon \gx & \epsilon \gy \end{pmatrix}.
\label{eq:blocks}
\end{equation}

\begin{theorem}[Cusped singularity, singular Hopf, and SAOs]
\label{thm:main}
Consider system \eqref{eq:full} with a symmetric non-degenerate cusped singularity $p^{*}=(x^{*},x^{*},y^{*},y^{*})$.
Assume:
\begin{itemize}
\item[(C1)] inhibitory coupling: $\ft<0$;
\item[(C2)] slow dynamics decay: $\gy \leq0$;
\item[(C3)] cross-derivative product: $\fy \gx < 0$;
\item[(C4)] cusp non-degeneracy: $\Gamma\neq 0$;
\item[(C5)] the cusp is not an equilibrium: $g_0^*= g(x^*,y^*)\neq 0$;
\item[(C6)] the flow crosses the cusp from the attracting sheet $\mathcal{S}_a$: 
$g_0^*\cdot  \Omega>0$ and $\Gamma >0$.
\end{itemize}
Then:
\begin{itemize}
    \item[(i)] 
    For $0<\epsilon\ll1$ the full system undergoes a singular Hopf bifurcation at a nearby
    parameter value satisfying $\fo - \ft = -\epsilon \gy$.
    The bifurcation is governed by the antisymmetric block $J_a$, with purely imaginary
    eigenvalues $\pm i\sqrt{\epsilon|\fy\gx|}+\mathcal{O}(\epsilon)$.
    \item[(ii)] For $0<\epsilon\ll1$, trajectories of \eqref{eq:full} that pass through a
    neighborhood of $p^*$ undergo SAOs. More precisely, the center manifold reduction of
    Lemma~\ref{lem:cm} yields a three-dimensional slow-fast system \eqref{eq:reduced_3d}
    whose critical manifold has a cusp singularity at $(u,w,z)=(0,0,0)$.
    Applying  
    the analysis of \cite{kristiansen2023}, the
    number of SAOs is
    \begin{equation}
    n_{\mathrm{SAO}} = \left\lfloor \frac{\lambda_2}{\lambda_1} \right\rfloor,
    \label{eq:SAO_count}
    \end{equation}
    provided $\lambda_2/\lambda_1 \notin \mathbb{N}$,
    where $\lambda_1 < \lambda_2 < 0$ are the eigenvalues of the desingularized reduced
    problem on $\mathcal{S}_a$ at the cusped singularity.
\end{itemize}
\end{theorem}

\begin{proof}
\textbf{Part (i).}
At the cusp, the geometric fold condition $\fo = \ft$ holds
(from \eqref{eq:fold_new}).
Because we assume inhibitory coupling $\ft < 0$, it follows that $\fo < 0$.
The conditions $\fo < 0$ and $\gy \leq 0$ alongside $\fy \gx < 0$  
guarantee
that the symmetric subspace is strongly attracting, since the symmetric block $J_s$ has
$\mathrm{tr}(J_s) = 2\fo + \epsilon\gy < 0$ and
$\det(J_s) \approx \epsilon(2\fo\gy - \fy\gx) > 0$ for 
positive $\epsilon$.
Turning to the antisymmetric block $J_a$, the trace vanishes when
$\fo - \ft = -\epsilon \gy$, placing a bifurcation at an $\mathcal{O}(\epsilon)$
parameter distance from the cusp (in the case $\gy=0$, at the cusp).
At this point,
$\det(J_a) = -(\epsilon \gy)^2 - \epsilon \fy \gx \approx -\epsilon \fy \gx$
for $0 < \epsilon \ll 1$.
The assumption $\fy \gx < 0$ ensures $\det(J_a) > 0$.
Because the trace vanishes while the determinant is strictly positive, the system undergoes
a Hopf bifurcation.
The corresponding eigenvalues are
$\lambda_{a,\pm} \approx \pm i \sqrt{\epsilon |\fy \gx|}$.
As these purely imaginary eigenvalues are of order $\mathcal{O}(\sqrt{\epsilon})$, this is
a singular Hopf bifurcation.

\smallskip
\textbf{Part (ii).}
By Lemma~\ref{lem:cm},  
under conditions (C1)--(C4), the extended system admits a smooth
attracting center manifold $M_a$ on which the dynamics reduces to the three-dimensional
system \eqref{eq:reduced_3d}.

First note that the transformation $y_i \mapsto \tilde y_i := - y_i$ transforms the system to one with $\tilde \fy = -\fy$, $\tilde \Omega = - \Omega$, $\tilde \Gamma=\Gamma$, $\tilde \gx=-\gx$, $\tilde \gy=\gy$, and $\tilde g_0^*=-g_0^*$. 
Hence all conditions (C1)-(C6) are invariant under this transformation, so 
without loss of generality we may assume
$\fy<0$. 
Then, by (C6), 
$\Omega$ and $g_0^*$ have the same sign whereas
$\Gamma$ and $\fy$ have opposite signs. 
By Remarks \ref{rem1} and \ref{rem2}, the cusp on critical manifold $\mathcal{S}$ opens up in the direction of negative 
(respectively positive) $w$ direction if $\Omega>0$ (respectively $\Omega<0$) and the center sheet is attracting. 
If $g_0^*>0$ (respectively $g_0^*<0$) the flow  at the cusp is towards positive (respectively, negative) $w$ values, 
i.e., 
the cusp is crossed  from the attracting sheet  of $\mathcal{S}$ since $\Omega$ and $g_0^*$ have the same sign.

The system \eqref{eq:reduced_3d} is slow-fast with one fast variable $u$ and two slow variables $(w,z)$.
By Step 5  
of the proof of Lemma~\ref{lem:cm}, under conditions (C3) and (C4) the
critical manifold $\mathcal{S}$ of \eqref{eq:reduced_3d} has a cusp singularity at the
origin. 
The cusp structure established in Step 5 of the proof of Lemma \ref{lem:cm}, with $(u,w,z)\sim(r,r^2,r^3)$
near the origin, determines the blowup  
\begin{equation}
(u,w,z,\epsilon) = (r\bar u,\, r^2\bar w,\, r^3\bar z,\, r^4\bar\epsilon)
\label{eq:blowup}
\end{equation}
as in \cite{kristiansen2023},
and the
analysis of \cite{kristiansen2023} applies directly.
Condition (C5) ensures $g_0^* \neq 0$, so that in the scaling chart $\bar\epsilon=1$
of the blowup \eqref{eq:blowup} the desingularized slow equation gives $\hat{\dot w}_2 = g_0^* \neq 0$, 
driving passage through the cusp region and yielding the Weber equation 
\begin{equation}
U_2'' - Y_2\,U_2' + \frac{\lambda_2}{\lambda_1}\,U_2 = 0
\label{eq:weber}
\end{equation}
for the linearization near $\gamma:\{u=z=0\}$.
By \cite[Lemma 7]{kristiansen2023}, whenever $\lambda_2/\lambda_1\notin\mathbb{N}$, the
attracting and repelling slow manifolds $\mathcal{S}_{a,\epsilon}$ and $\mathcal{S}_{r,\epsilon}$
intersect transversally along $\gamma$ with the tangent space of $\mathcal{S}_{a,\epsilon}$
twisting $\lfloor\lambda_2/\lambda_1\rfloor$ times.
Each twist corresponds to a full $180^\circ$ rotation in the $(u,z)$-plane, i.e., one SAO,
giving \eqref{eq:SAO_count}.
\end{proof}

\begin{remark}
The condition (C3)
reflects the typical negative feedback loop between
the fast and slow variables within individual units, representing, for instance, a fast
voltage variable activating a slow recovery variable ($\gx>0$) that in turn represses
the voltage ($\fy<0$).
\end{remark}

\begin{remark}
As is generic for singular Hopf bifurcations, MMOs appear from the SAOs combined with an
appropriate global return mechanism \cite{desroches2012}.
\end{remark}

\begin{remark}
The eigenvalues $\lambda_1,\lambda_2$ in \eqref{eq:SAO_count} are the eigenvalues of the
desingularized reduced flow on $\mathcal{S}_a$ at the cusp, obtained by multiplying
\eqref{eq:reduced_3d} at $\epsilon>0$ by $-\partial_u Q = (\Omega/\fy)w + \cdots$ and
linearizing.
For the FHN model they equal $\lambda_1=-6v_s(v_s-c)$ and $\lambda_2=-\lambda_1+2g$
\cite[eq.~(23)]{kristiansen2023}; for general systems they can be read off from the
analogous desingularized problem on $M_a$.
\end{remark}

\section{Cusped singularities in the Curtu model}
\label{sec:curtu}

The Curtu model \cite{curtu2010} describes two neural populations with mutual inhibition
\begin{align}
\dot{u}_{1} &= -u_{1}+S(I-bu_{2}-a_{1}+u_{1}), \quad
\dot{a}_{1} = \epsilon(-a_{1}+c u_{1}), \nonumber \\
\dot{u}_{2} &= -u_{2}+S(I-bu_{1}-a_{2}+u_{2}), \quad
\dot{a}_{2} = \epsilon(-a_{2}+c u_{2}), \label{eq:curtu}
\end{align}
where $S(x)=1/(1+\exp(-r(x-\theta)))$.
We fix the parameters to $I = 0.68$, $\epsilon = 0.01$, $b = 0.6055$, $c = 0.63$,
$r = 10$, and $\theta = 0.2$.
With these parameters, the system produces mixed-mode oscillations (MMOs), as shown in
Fig.~\ref{fig:curtu_mmo}A,B.

The critical manifold is given by $a_{i} = Y(u_{i},u_{j}) = I-bu_{j}+\phi(u_{i})$,
where $\phi(u) := u-S^{-1}(u)$.
The partial derivatives of $Y$ are given by $Y_1(u_i, u_j) = \phi^{\prime}(u_i) = 1-1/S'(\cdot)$  
and
$Y_2(u_i, u_j) = -b$.
The geometric fold condition for symmetric points \eqref{eq:fold_new} corresponds to
$Y_1(u^*, u^*) = Y_2(u^*, u^*)$, which yields
\begin{equation}
\phi^{\prime}(u^{*}) = 1- \frac{1}{ru^*(1-u^*)} = -b. \label{eq:curtu_fold}
\end{equation}
Evaluating this numerically for our parameter set gives the relevant upper-branch symmetric
fold point at $u^{*} \approx 0.933$.

To verify that this point forms a true cusp, we evaluate the non-degeneracy condition
\eqref{eq:cusp_new} by calculating the mixed partials of $Y$:
\begin{equation}
Y_{11}(u^*, u^*) - Y_{22}(u^*, u^*) = \phi^{\prime\prime}(u^{*}) - 0
= \frac{1-2u^*}{r (u^*)^2 (1-u^*)^2} \approx -22.2 \ne 0.
\end{equation}
Thus, all geometric conditions of Definition~\ref{def:cusp} are satisfied, establishing the
presence of a non-degenerate cusped singularity.
Moreover, 
since $\fy = -bS'(\cdot)<0$, $\gx = c>0$, and $\gy=-1$, the  
conditions  (C2) and (C3) of Theorem~\ref{thm:main} are  
satisfied.
Further, direct calculations give 
\[
\Gamma  
= \frac16 S'''(\cdot) (1+b)^3 - \frac1{4\fo} S''(\cdot)^2(1-b^2)(1+b)^2>0
\] since $S'''(\cdot)>0$ and $\fo=\ft <0$, and condition (C4) is also fulfilled.

Further, 
\begin{equation}
\Omega = S''(\cdot) [ -(1+b) - (1-b^2)\,\frac{\fy}{2\fo} ] >0 
\end{equation}
since $\fy<0$, $\fo<0$, and $S''(\cdot)<0$ which follows from the observation that $S^{-1}(u^*)>\theta$, i.e., the argument of $S$ lies to the right of its inflection point.

Thus, $\Omega/\fy < 0$ and $\Gamma/\Omega > 0$ as in
Remark~\ref{rem1}: the cusp opens in the $w<0$
direction and the central sheet of $\mathcal{S}$ is attracting ($\mathcal{S}_a$ at $w < 0$).
The slow dynamics evaluated at the cusped singularity gives
$g_0^* = g(u^*, a^*) \approx 0.0036 > 0$.
Hence (C5) is satisfied: the cusped singularity is not a full-system equilibrium.
Furthermore, $g_0^* \cdot \Omega > 0$ and
$\Gamma > 0$, so condition~(C6) is satisfied.
The positive $g_0^*$ drives $w$ upward ($\dot{w} \approx \epsilon g_0^* > 0$),
so trajectories approach the cusp from the attracting sheet $\mathcal{S}_a$
(located at $w < 0$), consistent with (C6) and Remark~\ref{rem1}.

The full-system equilibrium 
is a saddle-focus located at
$u_{eq} \approx 0.931$, which 
lies near but above the cusped singularity in the $a$
direction, on the repelling sheet $\mathcal{S}_r$ (since $w_{eq} = a_{\rm eq}-a^* \approx 0.0023 > 0$).
The proximity of the equilibrium to the fold allows the system to pass near the cusped singularity and enters the SAO regime before
escaping into LAOs, as described by Theorem~\ref{thm:main}(ii).
The geometry of the critical manifold and the corresponding cusp singularity are depicted
in Fig.~\ref{fig:curtu_mmo}C.

\begin{figure}[htbp]
    \centering
    \includegraphics[width=0.9\textwidth]{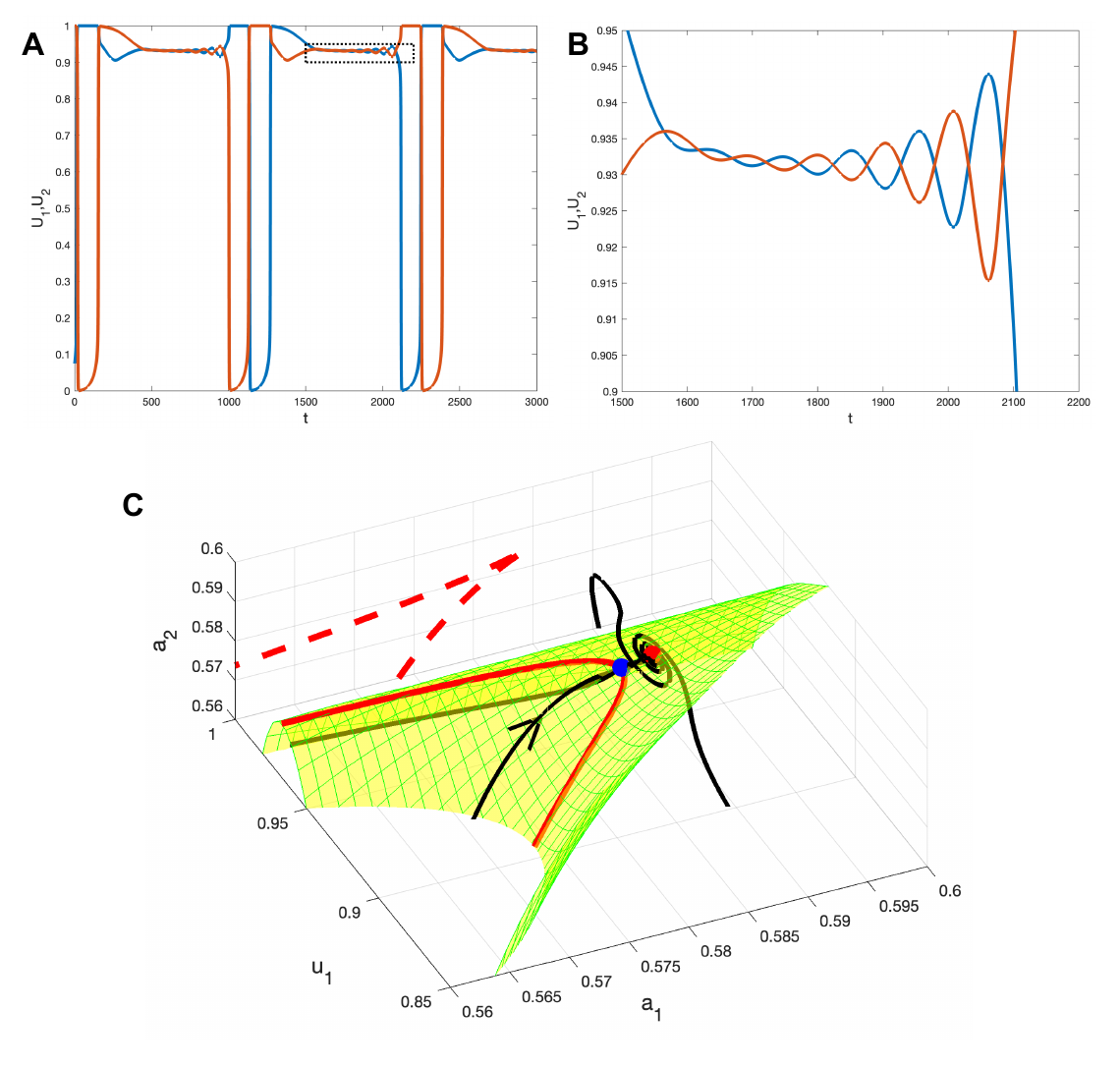}
    \caption{Mixed-mode oscillations and the cusped singularity in the Curtu model.
    (A) Time series showing MMOs for the chosen parameter set. Notice how cell 1 (red) and
    cell 2 (blue) alternate.
    (B) A zoomed-in view of the small amplitude oscillations (SAOs).
    (C) The critical manifold (yellow) geometry with the fold curve (red) revealing the
    cusp singularity (best seen in the projection of the fold onto the $(a_1,a_2)$ plane
    as the red, dashed curve) that organizes the dynamics. 
    The attracting sheet $\mathcal{S}_a$ is the part of the critical manifold to the left of the fold curve. 
    The trajectory corresponding
    to the simulation in panel (A) is shown in black. As it passes the cusp, SAOs are
    generated, and as the trajectory approaches the full-system saddle-focus where it
    spirals away, SAOs with increasing amplitude appear, eventually giving rise to the
    large-amplitude oscillations. Notice how the trajectory approaches the cusp region
    alternatingly from the left and from the right, explaining the alternating pattern
    seen in panel (A).}
    \label{fig:curtu_mmo}
\end{figure}

\section{Cusped singularities in coupled Morris-Lecar neurons}
\label{sec:ml}

Two Morris-Lecar neurons with fast synaptic inhibition \cite{morris1981} are described by
\begin{align}
C \dot{V}_{i} &= I_{app}-g_{Ca}m_{\infty}(V_{i})(V_{i}-V_{Ca})-g_{K}n_{i}(V_{i}-V_{K})
\nonumber \\
&\quad -g_{L}(V_{i}-V_{L})-g_{s}s_{\infty}(V_{j})(V_{i}-V_{syn}), \nonumber \\
\dot{n}_{i} &= \phi_{n} (n_{\infty}(V_{i})-n_{i})/\tau(V_{i}),
\quad i=1,2, \quad j\ne i. \label{eq:ml}
\end{align}
The activation functions and the time-scale function are given by
\begin{align}
m_{\infty}(V) &= 0.5 \left( 1+\tanh\left(\tfrac{V-v_1}{v_2}\right) \right), \nonumber \\
n_{\infty}(V) &= 0.5 \left( 1+\tanh\left(\tfrac{V-v_3}{v_4}\right) \right), \nonumber \\
\tau(V) &= 1/\cosh\left(\tfrac{V-v_3}{2v_4}\right),
\end{align}
and synaptic inhibition is modeled as
\begin{equation}
s_\infty(V) = \frac{1}{1+\exp(-(V-\theta_s)/k_s)}.
\end{equation}
With the parameters provided in Table~\ref{tab:ml_params}, the system produces MMOs
(Fig.~\ref{fig:ml_mmo}A,B).

\begin{table}[b]
\centering
\begin{tabular}{ccccccccc}
\hline
$C$ & $V_K$ & $g_K$ & $V_{Ca}$ & $g_{Ca}$ & $V_L$ & $g_L$ & $I_{app}$ & $V_{syn}$ \\
$(\mu\text{F/cm}^2)$ & (mV) & $(\text{mS/cm}^2)$ & (mV) & $(\text{mS/cm}^2)$ & (mV)
& $(\text{mS/cm}^2)$ & $(\mu\text{A/cm}^2)$ & (mV) \\
20 & -84 & 8 & 120 & 4.4 & -60 & 2 & 80 & -70 \\ \hline
$g_s$ & $\phi_n$ & $v_1$ & $v_2$ & $v_3$ & $v_4$ & $k_s$ & $\theta_s$ & \\
$(\text{mS/cm}^2)$ & $(\text{ms}^{-1})$ & (mV) & (mV) & (mV) & (mV) & (mV) & (mV) & \\
0.3 & 0.01 & -1.2 & 18 & 2 & 30 & 2 & -25 & \\ \hline
\end{tabular}
\caption{Parameters for the coupled Morris-Lecar neurons.}
\label{tab:ml_params}
\end{table}

Solving for the slow variable $n_{i}$ defines the critical manifold map
\begin{equation}
n_i = Y(V_{i},V_{j}) = \frac{f_{0}(V_{i})-g_{s}s_{\infty}(V_{j})(V_{i}-V_{syn})}{g_{K}(V_{i}-V_{K})},
\end{equation}
where $f_{0}(V)=I_{app}-g_{Ca}m_{\infty}(V)(V-V_{Ca})-g_{L}(V-V_{L})$.

The geometric fold condition \eqref{eq:fold_new} is given by setting
$Y_1(V^*, V^*) = Y_2(V^*, V^*)$.
Solving this equation numerically locates the symmetric fold point at $V^* \approx -30.36$ mV.
It is analytically more straightforward to verify the non-degeneracy condition using the
equivalent formulation \eqref{eq:cusp_f} instead of computing the second derivatives of
the quotient $Y$.
For the fast dynamics 
$f(V_i,V_j,n_i)=\bigl(f_0(V_i)-g_K n_i(V_i-V_K)-g_s s_\infty(V_j)(V_i-V_{\mathrm{syn}})\bigr)/C$,
the required partial derivatives evaluated at the symmetric fold $V^*$ are:
\begin{equation}\label{eq:ML_derivatives}
\begin{aligned}
f_1^* &= \bigl(f_0'(V^*)-g_K n^*-g_s s_\infty(V^*)\bigr)/C, & 
f_2^* &= -g_s s_\infty'(V^*)(V^*-V_{\mathrm{syn}})/C, \\
f_{11}^* &= f_0''(V^*)/C, & 
f_{12}^* &= -g_s s_\infty'(V^*)/C, \\
f_{22}^* &= -g_s s_\infty''(V^*)(V^*-V_{\mathrm{syn}})/C, & 
f_{111}^* &= f_0'''(V^*)/C, \\
f_{112}^* &= 0, & 
f_{122}^* &= -g_s s_\infty''(V^*)/C, \\
f_{222}^* &= -g_s s_\infty'''(V^*)(V^*-V_{\mathrm{syn}})/C, &
f_y^* &= -g_K(V^*-V_K)/C, \\ 
f_{1y}^* &= -g_K/C,& \quad f_{2y}^* &= 0.
\end{aligned}
\end{equation}
Substituting $f_{11}^*$, $f_{22}^*$, $f_{1y}^*$ and $f_{2y}^*$ into (7), the explicit geometric 
non-degeneracy constraint evaluates as
\begin{equation}\label{eq:ML_Dstar}
D^* = \frac{f_0''(V^*)}{C} + \frac{g_s s_\infty''(V^*)(V^*-V_{\mathrm{syn}})}{C} 
- \frac{2f_1^*}{V^*-V_K} \approx 0.02 > 0.
\end{equation}
Thus, we have a non-degenerate cusped singularity.

Clearly conditions (C1)--(C2) of Theorem~\ref{thm:main} are fulfilled. 
Since $g_x^*=\phi_n/\tau(V^*)>0$ and $f_y^*<0$, we have $f_y^* g_x^*<0$, so condition (C3) 
is satisfied. Moreover, substituting the derivatives \eqref{eq:ML_derivatives} into the 
center-manifold coefficient (15) yields
\begin{multline}\label{eq:ML_A_corrected}
\Gamma = \frac{1}{6C}\Bigl[f_0'''(V^*)-3g_s s_\infty''(V^*)
+g_s s_\infty'''(V^*)(V^*-V_{\mathrm{syn}})\Bigr] \\
- \frac{\bigl[f_0''(V^*)+g_s s_\infty''(V^*)(V^*-V_{\mathrm{syn}})\bigr]
\bigl[f_0''(V^*)+2g_s s_\infty'(V^*)-g_s s_\infty''(V^*)(V^*-V_{\mathrm{syn}})\bigr]}
{4C^2 f_1^*}.
\end{multline}
A direct numerical evaluation gives $\Gamma\approx 0.002>0$, confirming condition~(C4).

At the fold $\fo=\ft <0$ and since also $\fy<0$, $\Omega$ and $D^*$ have opposite signs, i.e., $\Omega<0$. 
Hence, $\Omega/\fy>0$ and $\Gamma/\Omega < 0$, and the cusp opens in
the $w > 0$ direction with the central sheet being attracting, see Remark~\ref{rem2}, in contrast to the Curtu model.

At the cusped singularity, $n^* = Y(V^*, V^*) \approx 0.1046$ while
$n_\infty(V^*) \approx 0.1036$, so $n^* > n_\infty(V^*)$ and $g_0^*<0$.
Hence (C5) is satisfied.
Furthermore, $g_0^* \, \Omega  > 0$ and $\Gamma > 0$, so condition~(C6) is satisfied.
The negative $g_0^*$ drives $w$ downward ($\dot{w} \approx \epsilon g_0^* < 0$),
so trajectories approach the cusp from the attracting sheet $\mathcal{S}_a$, consistent with (C6) and Remark~\ref{rem2}.
The full-system equilibrium (a saddle-focus), found numerically at $V_{\rm eq} \approx -30.24$~mV with
$n_{\rm eq} = n_\infty(V_{\rm eq}) \approx 0.1044$, lies at
$n_{\rm eq} - n^* \approx -0.0002 < 0$, i.e., below the cusp in the $n$-direction,
on the repelling sheet $\mathcal{S}_r$.

All conditions of Theorem~\ref{thm:main} are satisfied, and the system enters the
characteristic SAO regime as described by Theorem~\ref{thm:main}(ii).
The geometry of the critical manifold with the cusp singularity and the corresponding
full-system saddle-focus are shown in Fig.~\ref{fig:ml_mmo}C.

\begin{figure}[tbp]
    \centering
    \includegraphics[width=0.9\textwidth]{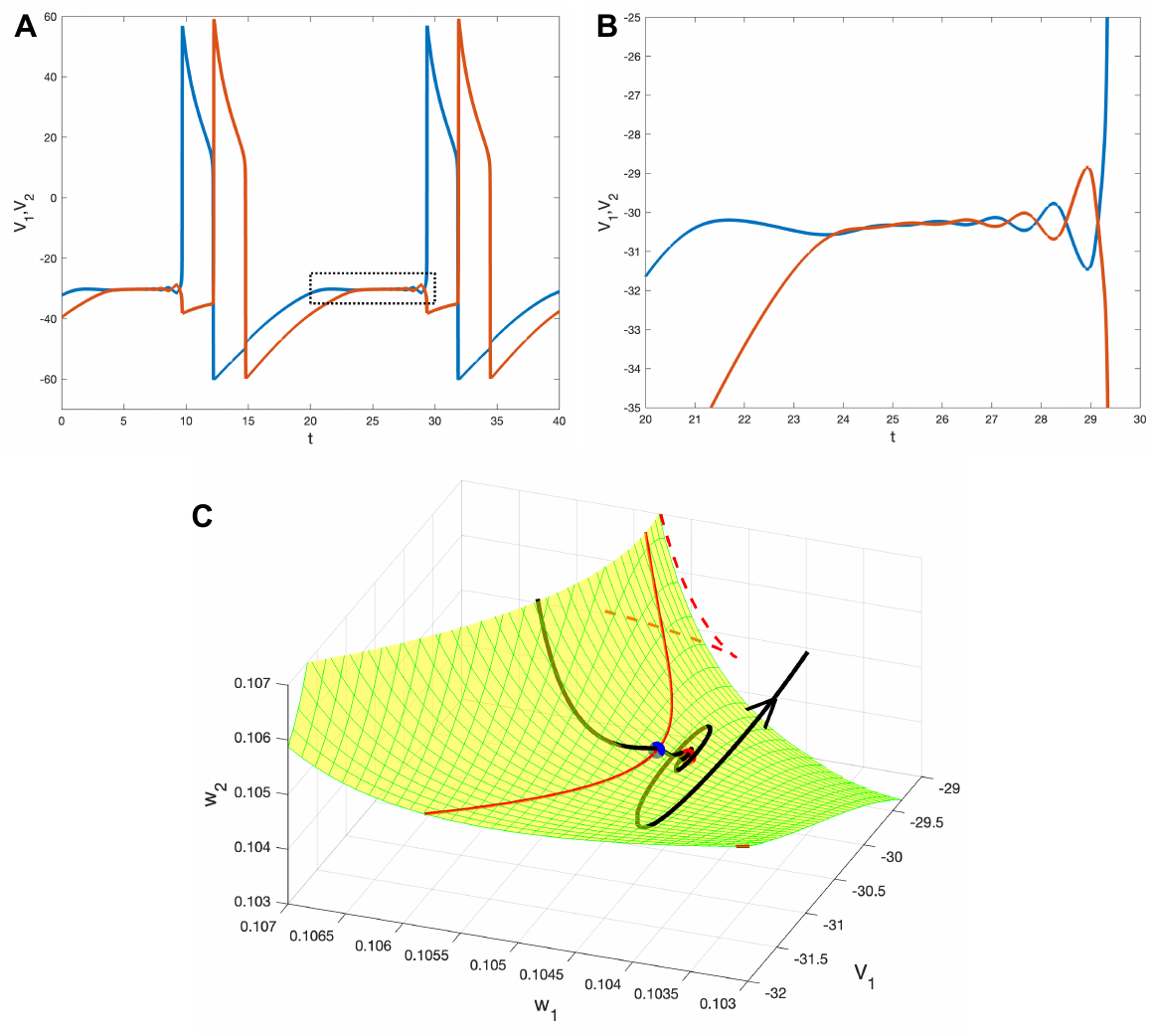}
    \caption{Mixed-mode oscillations and the cusped singularity in the coupled
    Morris-Lecar model. (A) Time series showing MMOs for the chosen parameter set. (B) A
    zoomed-in view of the MMOs showing the alternating small and large amplitude
    oscillations. (C) The critical manifold geometry revealing the cusp singularity that
    organizes the dynamics. 
    The attracting sheet $\mathcal{S}_a$ is the part of the critical manifold to the left of the fold curve. 
    Colors and symbols as in Fig.~\ref{fig:curtu_mmo}C}
    \label{fig:ml_mmo}
\end{figure}

\section{Discussion}
\label{sec:discussion}

We have established cusped singularities as a universal mechanism for mixed-mode
oscillations in mutually inhibitory slow-fast systems.
Our main results, Lemma~\ref{lem:cm} and Theorem~\ref{thm:main}, together provide a
complete account: Lemma~\ref{lem:cm} carries out the center manifold reduction and
identifies the precise conditions on $f$ and $g$ under which the blow-up analysis of
\cite{kristiansen2023} transfers from the FHN model to the general class \eqref{eq:full},
while Theorem~\ref{thm:main} combines this with the singular Hopf analysis to give both
the qualitative (SAOs occur) and quantitative ($n_{\mathrm{SAO}}=\lfloor\lambda_2/\lambda_1\rfloor$)
description of the dynamics.
In most models, the non-degeneracy conditions $D^*\neq 0$ and $\Gamma\neq 0$ will be fulfilled, which we confirmed explicitly for  the Curtu and Morris-Lecar models.

The result that the existence of a cusped singularity in the critical manifold geometry yields a
singular Hopf bifurcations in the full dynamics provides a powerful diagnostic tool: the
existence of MMOs can be predicted by analyzing the shape of the critical manifold in
combination with its proximity to the full system equilibrium.
Specifically, as shown  
in both the Curtu and Morris-Lecar systems,
actively tuning system parameters to push the full-system symmetric equilibrium near the cusped
singularity  induces SAOs related to the nearby singular Hopf bifurcation.
MMOs are created when the SAOs are combined with an appropriate return mechanism that reinjects the system into the cusp region. Due to the symmetry in the system, and depending on the parameters, the reinjection can occur from either side of the cusp leading to MMO patterns where the two single units alternate \cite{kristiansen2023}, as confirmed here for the Curtu model (Fig.~\ref{fig:curtu_mmo}). 
This behavior is in contrast to standard folded-node controlled MMOs where the entry into the node funnel always happens from the same side \cite{desroches2012}.

The presented  framework  applies broadly and beyond neuroscience to symmetric systems featuring inhibitory coupling, such as chemical oscillators coupled via a common inhibitor \cite{vanag2010}.
Future work should explore higher-order systems and asymmetric perturbations of \eqref{eq:full}.
Experimental investigation of MMOs in inhibitory neural cultures could test whether the
distinctive signatures predicted here are observed in biological systems.

\vspace{0.5cm}
\noindent \textbf{Declaration of generative AI} \\
During the preparation of this work, the author used AI-assisted tools (Gemini 3.1 Pro,
Kimi K2.6 Thinking, Claude Sonnet 4.6) for performing and checking 
calculations, as well as for drafting and correcting text. After using these tools, the
author reviewed and edited the content as needed and takes full responsibility for the
content of the published article.


\begin{thebibliography}{99}
    \bibitem{desroches2012} M. Desroches, J. Guckenheimer, B. Krauskopf, C. Kuehn,
    H. M. Osinga, M. Wechselberger, Mixed-mode oscillations with multiple time scales,
    SIAM Rev. 54 (2012) 211--288. DOI: 10.1137/100791233.
    \bibitem{wechselberger2005} M. Wechselberger, Existence and bifurcation of canards
    in $\mathbb{R}^3$ in the case of a folded node, SIAM J. Appl. Dyn. Syst. 4 (2005)
    101--139. DOI: 10.1137/030601995.
    \bibitem{krupa2010} M. Krupa, M. Wechselberger, Local analysis near a folded
    saddle-node singularity, J. Differ. Equ.  248 (2010) 2841--2888.
    DOI: 10.1016/j.jde.2010.02.006.
    \bibitem{guckenheimer2008} J. Guckenheimer, Singular Hopf bifurcation in systems
    with two slow variables, SIAM J. Appl. Dyn. Syst. 7 (2008) 1355--1377.
    DOI: 10.1137/080718528.
    \bibitem{battaglin2021} S. Battaglin, M. G. Pedersen, Geometric analysis of mixed-mode oscillations 
    in a model of electrical activity in human beta-cells, Nonlinear Dyn. 104 (2021) 4445--4457.
    DOI: 10.1007/s11071-021-06514-z.
    \bibitem{kristiansen2023} K.U. Kristiansen, M.G. Pedersen, Mixed-mode oscillations
    in coupled FitzHugh-Nagumo oscillators: Blow-up analysis of cusped singularities,
    SIAM J. Appl. Dyn. Syst. 22 (2023) 1383--1422. DOI: 10.1137/22M1480495.
    \bibitem{pedersen2022} M. G. Pedersen, M. Br{\o}ns, M. P. S{\o}rensen,
    Amplitude-modulated spiking as a novel route to bursting: Coupling-induced mixed-mode
    oscillations by symmetry breaking, Chaos 32 (2022) 013121.
    DOI: 10.1063/5.0072497.
    \bibitem{curtu2010} R. Curtu, Singular Hopf bifurcations and mixed-mode oscillations
    in a two-cell inhibitory neural network, Phys. D 239 (2010) 504--514.
    DOI: 10.1016/j.physd.2009.12.010.
    \bibitem{morris1981} C. Morris, H. Lecar, Voltage oscillations in the barnacle giant
    muscle fiber, Biophys. J. 35 (1981) 193--213.
    DOI: 10.1016/S0006-3495(81)84782-0.
    \bibitem{carr1981} J. Carr, Applications of Centre Manifold Theory, Appl. Math. Sci.
    35, Springer-Verlag, New York, 1981.
    \bibitem{vanag2010} V.K. Vanag, I.R. Epstein,  Periodic perturbation of one of two identical 
    chemical oscillators coupled via inhibition, Phys. Rev. E 81(2010) 066213.
    DOI: 10.1103/PhysRevE.81.066213
\end{thebibliography}
\end{document}